\documentclass[12pt]{amsart}
\usepackage{enumerate, amsmath, amsthm, amsfonts, amssymb, mathrsfs, graphicx, paralist}
\usepackage[all]{xy}
\usepackage[usenames, dvipsnames]{color}
\usepackage[margin=1in]{geometry} 
\usepackage[bookmarks, bookmarksdepth=2, colorlinks=true, linkcolor=blue, citecolor=blue, urlcolor=blue]{hyperref}
\usepackage{eucal}
\usepackage{verbatim}
\usepackage{tikz}
\usetikzlibrary{automata,positioning}
\numberwithin{equation}{section}
\newtheorem{theorem}[equation]{Theorem}

\newtheorem{proposition}[equation]{Proposition}
\newtheorem{lemma}[equation]{Lemma}
\newtheorem{corollary}[equation]{Corollary}

\theoremstyle{definition}
\newtheorem{rmk}[equation]{Remark}
\newenvironment{remark}[1][]{\begin{rmk}[#1] \pushQED{\qed}}{\popQED \end{rmk}}
\newtheorem{eg}[equation]{Example}
\newenvironment{example}[1][]{\begin{eg}[#1] \pushQED{\qed}}{\popQED \end{eg}}
\newtheorem{defn}[equation]{Definition}

\newcommand{\arxiv}[1]{\href{http://arxiv.org/abs/#1}{{\tt arXiv:#1}}}

\newcommand{\bm}{\mathbf{m}}
\newcommand{\cI}{\mathcal{I}}
\newcommand{\cL}{\mathcal{L}}
\newcommand{\cM}{\mathcal{M}}
\newcommand{\cA}{\mathcal{A}}
\newcommand{\bN}{\mathbf{N}}
\newcommand{\bw}{\mathbf{w}}
\newcommand{\BF}[1]{\mathbf #1}

\DeclareMathOperator{\Aut}{Aut}
\DeclareMathOperator{\len}{len}
\DeclareMathOperator{\std}{std}
\DeclareMathOperator{\Inc}{Inc}
\DeclareMathOperator{\ini}{in}
\DeclareMathOperator{\Id}{Id}

\title[Hilbert series via formal languages]{Hilbert series of symmetric ideals\\ 
in infinite polynomial rings\\ 
via formal languages}
\date{\today}

\author{Robert Krone}
\address{Department of Mathematics and Statistics, Queen's University, Kingston, ON}
\email{\href{mailto:rk71@queensu.ca}{rk71@queensu.ca}}
\urladdr{\url{http://rckr.one/}}

\author{Anton Leykin}
\address{School of Mathematics, Georgia Institute of Technology, Atlanta, GA}
\email{\href{mailto:leykin@math.gatech.edu}{leykin@math.gatech.edu}}
\urladdr{\url{http://people.math.gatech.edu/~aleykin3/}}

\thanks{AL was supported by NSF grant DMS-1151297}

\author{Andrew Snowden}
\address{Department of Mathematics, University of Michigan, Ann Arbor, MI}
\email{\href{mailto:asnowden@umich.edu}{asnowden@umich.edu}}
\urladdr{\url{http://www-personal.umich.edu/~asnowden/}}

\thanks{AS was supported by NSF grants DMS-1303082 and DMS-1453893}
\begin{document}

\begin{abstract}
Let $R$ be the polynomial ring $K[x_{i,j}]$ where $1 \le i \le r$ and $j \in \bN$, and let $I$ be an ideal of $R$ stable under the natural action of the infinite symmetric group $S_{\infty}$. Nagel--R\"omer recently defined a Hilbert series $H_I(s,t)$ of $I$ and proved that it is rational. We give a much shorter proof of this theorem using tools from the theory of formal languages and a simple algorithm that computes the series.
\end{abstract}

\maketitle
\tableofcontents

\section{Introduction}

\subsection{Statement of results}

Let $R$ be the polynomial ring over the field $K$ in variables $x_{i,j}$, where $i \in \{1,\ldots,r\}$ and $j \in \bN$. The infinite symmetric group $S_{\infty}$ acts on $R$ (by fixing the first index and moving the second), and a fundamental result, proved originally by Cohen \cite{Cohen} but subsequently rediscovered \cite{AschenbrennerHillar,HillarSullivant}, is that $R$ is $S_{\infty}$-noetherian: that is, any $S_{\infty}$-ideal in $R$ is generated by the $S_{\infty}$-oribts of finitely many elements. Given this, one can begin to study finer properties of ideals. In this paper, we investigate their Hilbert series.

Let $I \subset R$ be a homogeneous $S_{\infty}$-ideal. For $n \ge 1$, let $R_n \subset R$ be the subalgebra generated by the variables $x_{i,j}$ with $1 \le i \le r$ and $j \le n$, and put $I_n = I \cap R_n$. Then $I_n$ is a finitely generated graded $R_n$-module, and so its Hilbert series $H_{I_n}(t)$ is a well-defined rational function. We define the Hilbert series of $I$ by
\begin{displaymath}
H_I(s,t) = \sum_{n \ge 0} H_{I_n}(t) s^n.
\end{displaymath}
This series was introduced by Nagel--R\"omer \cite{NagelRoemer}, who proved the following theorem:

\begin{theorem} \label{mainthm}
The series $H_I(s,t)$ is a rational function of $s$ and $t$.
\end{theorem}

The purpose of this paper is to give a new proof of this theorem. Our proof is shorter and (in our opinion) conceptually clearer than the one given in \cite{NagelRoemer}.

\begin{remark}
In fact, \cite{NagelRoemer} work with what we would call $H_{R/I}(s,t)$, but it is a trivial matter to pass between this and our $H_I(s,t)$.
\end{remark}

\begin{remark}
The result of \cite{NagelRoemer} gives information about the denominator of $H_I(s,t)$. Our method gives some information as well, though we have not carefully traced through everything to see exactly what it yields. In particular, we do not know which method will ultimately say more about the denominator.
\end{remark}

\subsection{Overview of proof}

We now describe the idea of our proof. First, passing to the initial ideal one can reduce to the case where $I$ is a monomial ideal. One then has what is essentially a complicated bookkeeping problem: one must understand which of the monomials in the infinitely many variables $x_{i,j}$ appear in $I$. Our main idea is to use a sort of encoding scheme to make the problem more finite: more precisely, we establish a bijection between the monomials in $R$ and a certain set of words in a finite alphabet. Thus, in a sense, we trade the infinitely many commuting variables of $R$ for finitely many non-commuting variables. We show that, under this encoding scheme, $I$ (or rather, the set of monomials it contains) corresponds to a regular language. The theorem then follows from standard results on generating functions of regular languages.

The idea of using formal languages was motivated by the approach to Hilbert series in \cite{catgb}. However, the result and methods of this paper do not appear to fit into the general setup of \cite{catgb}.

\subsection{Outline}

In \S \ref{s:background} we review background material on regular languages. In \S \ref{s:monomial} we prove the main theorem in the case of monomial ideals; this is really the bulk of the work. In \S \ref{s:general} we complete the proof of the theorem by reducing to the monomial case. In \S \ref{s:algorithm} we explicitly describe an algorithm for computing $H_I(s,t)$, given a set of generators for $I$. Finally, in \S \ref{s:module} we discuss the possibility of treating Hilbert series of $R$-modules.

\subsection{Notation}

We write $\bN$ for the set of non-negative integers. We let $\Inc(\bN)$ be the so-called increasing monoid: this is the set of functions $f \colon \bN \to \bN$ satisfying $f(n)<f(m)$ for $n<m$, using composition as the monoidal operation. Throughout, $K$ denotes an arbitrary field.

\section{Background on regular languages} \label{s:background}

In this section we review some well-known material on formal languages, especially regular languages. We refer the reader to the text \cite{hopcroftullman} for more details.

Let $\Sigma$ be a finite set and let $\Sigma^{\star}$ be the set of words in the alphabet $\Sigma$; alternatively, $\Sigma^{\star}$ is the free monoid on $\Sigma$. A {\bf formal language} on $\Sigma$ is simply a subset of $\Sigma^{\star}$. Given a formal language $\cL$ on $\Sigma^{\star}$, we define the {\bf Kleene star} $\cL^{\star}$ of $\cL$ to be the language consisting of all words of the form $w_1 \cdots w_n$ where $w_i \in \cL$; alternatively, $\cL^{\star}$ is the submonoid of $\Sigma^{\star}$ generated by $\cL$. Given two formal languages $\cL_1$ and $\cL_2$, we define their {\bf concatenation} $\cL_1 \cL_2$ to be the formal language consisting of all words of the form $w_1 w_2$ with $w_1 \in \cL_1$ and $w_2 \in \cL_2$. We also make use of the standard set-theoretic operations of union, intersection, and complement on formal languages.

The class of {\bf regular languages} on $\Sigma$ is the smallest class of languages containing the singleton languages $\{\sigma\}$ for each $\sigma \in \Sigma$, and closed under union, concatenation, and Kleene star. (Actually, the empty language and the language consisting only of the empty word are also counted as regular languages, but do not fit the previous definition.) It turns out that the class of regular languages is also closed under intersection and complement.

Let $t_1, \ldots, t_k$ be a set of formal variables, let $\cM$ be the set of monomials in these variables, and let $\rho \colon \Sigma^{\star} \to \cM$ be a monoid homomorphism, which we refer to as the weight function. We note that $\rho$ is determined by its restriction to $\Sigma$. Given a language $\cL$ on $\Sigma$, we define its {\bf generating function} with respect to $\rho$ by
\begin{displaymath}
H_{\cL,\rho}(t_1, \ldots, t_k) = \sum_{w \in \cL} \rho(w),
\end{displaymath}
assuming this sum makes sense (i.e., there are only finitely many $w \in \cL$ for which $\rho(w)$ is a given monomial). We consider this as a formal power series in the variables $t_1, \ldots, t_k$. For example, suppose $k=1$ and $\rho$ is defined by $\rho(\sigma)=t$ for all $\sigma \in \Sigma$. Then for a word $w$ we have $\rho(w)=t^{\len(w)}$, and so the coefficient of $t^n$ in $H_{\cL,\rho}(t)$ is the number of words in $\cL$ of length $n$. We require the following standard result (see, e.g., \cite[Theorem~4.7.2]{stanley}, though the terminology there is somewhat different):

\begin{proposition}\label{prop:reggen}
If $\cL$ is a regular language then $H_{\cL,\rho}(t_1, \ldots, t_k)$ is a rational function of the $t_i$'s, for any weight function $\rho$ (for which the series makes sense).
\end{proposition}
\section{Monomial ideals} \label{s:monomial}

Let $R=K[x_{i,j}]$ where $1 \le i \le r$ and $j \in \bN$, and let $\cM$ be the set of monomials in $R$. Let $\Sigma$ be the alphabet $\{\tau, \xi_1, \ldots, \xi_r\}$. Let $T \colon \cM \to \cM$ be the shift operator, defined by $T(x_{i,j})=x_{i,j+1}$ and extended multiplicatively. We define a function $\bm \colon \Sigma^{\star} \to \cM$ inductively using the following three rules: (a) $\bm(\emptyset)=1$; (b) $\bm(\xi_i w)=x_{i,0} \cdot \bm(w)$; and (c) $\bm(\tau w)=T(\bm(w))$. Thus, concretely, to compute $\bm(w)$ simply change each $\xi_i$ in $w$ to $x_{i,0}$ and each $\tau$ to $T$ applied to the string following it.

\begin{example}
We have
$\bm(\tau \xi_1 \tau \xi_2 \tau) = T(x_{1,0}T(x_{2,0}T(1))) = T(x_{1,0} x_{2,1}) = x_{1,1} x_{2,2}.$
\end{example}

It is clear that the map $\bm \colon \Sigma^{\star} \to \cM$ is surjective, though it is not injective since the variables $x_{i,j}$ commute, e.g., $\bm(\xi_1 \xi_2)=\bm(\xi_2 \xi_1)$. We therefore introduce a subset of $\Sigma^{\star}$ to obtain a bijection. We say that a word $w$ in $\Sigma^{\star}$ is {\bf standard} if it satisfies the condition that every substring $\xi_i \xi_j$ of $w$ has $i \le j$. Let $\Sigma^{\star}_{\std}$ be the set of standard words, and let $\Sigma^{\star}_{\std,n}$ be the set of standard words in which $\tau$ occurs exactly $n$ times. Let $\cM_n$ be the set of monomials in the variables $x_{i,j}$ with $1 \le i \le r$ and $0 \le j \le n$.

\begin{proposition}
For each $n$ the map $\bm \colon \Sigma^{\star}_{\std,n} \to \cM_n$ is a bijection.
\end{proposition}

\begin{proof}
Let $u$ and $w$ be words in $\Sigma^{\star}_{\std,n}$ such that $\bm(u)=\bm(w)$, and let us prove $u=w$. Let $u'$ be the segment of $u$ appearing before the first $\tau$ in $u$, and write $u=u' u''$; similarly decompose $w=w' w''$. Note that $u'$, $u''$, $w'$, and $w''$ are all standard. Every variable in $\bm(u')$ has second index equal to~0, while every variable in $\bm(u'')$ has second index greater than~0, and similarly for $\bm(w')$ and $\bm(w'')$. We have
\begin{displaymath}
\bm(u') \bm(u'') = \bm(u) = \bm(w) = \bm(w') \bm(w'')
\end{displaymath}
and so $\bm(u')=\bm(w')$ and $\bm(u'')=\bm(w'')$. Since $u'$ and $w'$ are standard, it is clear that $u'=w'$.  If $n=0$ then $u''$ and $w''$ are empty and thus equal.  If $n > 0$ then $u'' = \tau u'''$ and $w'' = \tau w'''$ and $u''',w''' \in \Sigma^{\star}_{\std,n-1}$.  Since $T$ is injective on $\cM$, we have $\bm(u''')=\bm(w''')$.  By induction on $n$, $u''' = w'''$, thus $u=w$. We have thus shown that $\bm \colon \Sigma^{\star}_{\std} \to \cM$ is injective; it is clearly surjective.
\end{proof}

We let $\bw \colon \cM \to \Sigma^{\star}_{\std}$ be the right-inverse to the map $\bm$ which sends monomial $m$ to the minimal length word $w$ such that $\bm(w) = m$.  The image of $\bw$ is the set of words in $\Sigma^{\star}_{\std}$ that do not end in $\tau$.  On the other hand $\bm^{-1}(m) = w\tau^*$, the set of words consisting of $w$ followed by any number of trailing $\tau$s.

Given a monomial $m \in \cM$, let $\langle m \rangle$ be the set of monomials $m' \in \cM$ such that $\sigma(m) \mid m'$ for some $\sigma \in \Inc(\bN)$. Given monomials $m_1, \ldots, m_n$, let $\langle m_1, \ldots, m_n \rangle$ be the union of the $\langle m_i \rangle$'s.

\begin{proposition}\label{prop:Igfunc}
 Let $m_1,\ldots,m_n$ be monomials in $R$ and let $I$ the monomial ideal generated by the $\Inc(\bN)$-orbits of $m_1,\ldots,m_n$.  Let $\rho$ be the weight function defined by $\rho(\tau) = s$ and $\rho(\xi_i) = t$ for all $i = 1,\ldots,r$.  Then
 \[ H_I(s,t) = H_{\bm^{-1}(\langle m_1,\ldots,m_n\rangle),\rho}(s,t). \]
\end{proposition}
\begin{proof}
 Let $\cI \subset \cM$ be the set of monomials in $I$. Then $\cI=\langle m_1, \ldots, m_n \rangle$.  The coefficient of $s^n t^m$ in $H_I(s,t)$ is the number of monomials in $\cI \cap \cM_n$ of degree $m$. This equals the number of words in $\bm^{-1}(\cI)$ in which $\tau$ appears exactly $n$ times and which contain exactly $m$ non-$\tau$ letters. But this is just the coefficient of $s^n t^m$ in $H_{\bm^{-1}(\cI),\rho}(s,t)$ as defined at the end of \S \ref{s:background}. Thus $H_I(s,t)=H_{\bm^{-1}(\cI),\rho}(s,t)$, and so the result follows from Proposition~\ref{prop:reggen}.
\end{proof}

We say that a word in $\Sigma^{\star}$ is {\bf simple} if it contains no $\tau$.

\begin{proposition}
The set $\Sigma^{\star}_{\std}$ is a regular language on $\Sigma$.
\end{proposition}

\begin{proof}
Let $\cL$ be the language of simple standard words. The identity
\begin{displaymath}
\cL = \{ \xi_1 \}^{\star} \cdot \{ \xi_2 \}^{\star} \cdots \{ \xi_n \}^{\star}
\end{displaymath}
shows that $\cL$ is regular. The identity
$\Sigma^{\star}_{\std} = \cL \cdot (\tau \cL)^{\star}$
now shows that $\Sigma^{\star}_{\std}$ is regular.
\end{proof}

%

\begin{proposition} \label{prop:monreg}
Let $m \in \cM$. Then $\bm^{-1}(\langle m \rangle)$ is a regular language on $\Sigma$.
\end{proposition}

\begin{proof}
Write $\bw(m)=w_0 \tau w_1 \tau \cdots \tau w_n$, where each $w_i$ is simple. Let $\cL_i$ be the language consisting of simple standard words $w'$ such that $\bm(w_i) \mid \bm(w')$. One easily sees that $\cL_i$ is a regular language. Let $\cL$ be the regular language on $\Sigma$ defined by
\begin{equation}
\label{eq:regex}
\Sigma^{\star} \cL_0 \Sigma^{\star} \tau \cL_1 \Sigma^{\star} \tau \cL_2 \cdots \Sigma^{\star} \tau \cL_n \Sigma^\star.
\end{equation}
We claim that a monomial $m'$ belongs to $\langle m \rangle$ if and only if $\bw(m') \in \cL$. This will prove the proposition, as then $\bm^{-1}(\langle m \rangle)$ will coincide with $\cL \cap \Sigma^{\star}_{\std}$, and $\Sigma^{\star}_{\std}$ is also regular.

First suppose $m' \in \langle m \rangle$, so that $\sigma(m) \mid m'$ for some $\sigma \in \Inc(\bN)$. Write $m=m_0 \cdots m_n$ where $m_j$ uses only the variables $x_{i,j}$, and similarly write $m'=m'_0 \cdots m'_t$. Then $\sigma(m_j) \mid m'_{\sigma(j)}$ for $0 \le j \le n$. We have $\bw(m')=w_0' \tau w_1' \tau \cdots \tau w_t'$ where $\tau^j w_j'=\bw(m_j')$. We can regroup this expression as
\begin{displaymath}
\bw(m') = (\cdots) w'_{\sigma(0)} (\cdots) \tau w'_{\sigma(1)} (\cdots) \cdots (\cdots) \tau w'_{\sigma(n)} (\cdots)
\end{displaymath}
Since $\tau^{\sigma(j)} w_j = \bw(\sigma(m_j))$ and $\sigma(m_j) \mid m'_{\sigma(j)}$, we see that $\bm(w_j) \mid \bm(w'_{\sigma(j)})$ and so $w'_{\sigma(j)} \in \cL_j$. Thus the above expression shows that $\bw(m') \in \cL$.  Finally, if $\bw(m')$ is in $\cL$ then so is the set $\bw(m')\tau^{\star} = \bm^{-1}(m')$. 

Now suppose $w' \in \cL$. Write $\bm(w')=m_0' \cdots m'_t$ and $w'=w_0' \tau \cdots \tau w'_t\tau^k$ as above. Since $\bw(m') \in \cL$, we can find $\sigma(0)<\sigma(1)<\cdots<\sigma(n)$ such that $w'_{\sigma(j)} \in \cL_j$ for $0 \le j \le n$. Extend $\sigma$ arbitrarily to an element of $\Inc(\bN)$. Then it is clear that $\sigma(m) \mid \bm(w')$, and so $\bm(w') \in \langle m \rangle$.
\end{proof}

\begin{corollary}\label{cor:monreg}
Let $m_1, \ldots, m_n \in \cM$. Then $\bm^{-1}(\langle m_1, \ldots, m_n \rangle)$ is a regular language on $\Sigma$.
\end{corollary}

\begin{theorem} \label{thm:monomial}
Let $I \subset R$ be an $\Inc(\bN)$-stable monomial ideal. Then $H_I(s,t)$ is a rational function.
\end{theorem}

\begin{proof}
It is known (see \cite{Cohen} or \cite{HillarSullivant}) that $I$ is finitely generated up to the action of $\Inc(\bN)$: that is, there exist $m_1, \ldots, m_n \in I$, which can be taken to be monomials, such that $I$ is the ideal generated by the $\Inc(\bN)$-orbits of $m_1, \ldots, m_n$.  By Propositions \ref{prop:Igfunc} and \ref{prop:reggen}, $H_I(s,t)$ is rational if $\bm^{-1}(\langle m_1, \ldots, m_n \rangle)$ is a regular language, which is the result of Corollary \ref{cor:monreg}.
\end{proof}

\begin{remark}
 The above construction can be generalized from the total degree grading to arbitrary $\Inc(\bN)$-stable (multi-) grading.  An $\Inc(\bN)$-stable multi-grading, $\deg:\cM \to \mathbb{Z}^k$, is determined by the values of $\deg(x_{i,0})$ for $i = 1,\ldots,r$.  The series $H_I$ is then given by
  \[ H_I(s,t_1,\ldots,t_k) = H_{\Sigma^{\star}_{\std},\rho}(s,t_1,\ldots,t_k) \]
 for weight function $\rho$ with $\rho(\tau) = s$ and $\rho(\xi_i) = \deg(x_{i,0})$ for $i = 1,\ldots,r$.
\end{remark}

\section{General ideals} \label{s:general}

Let $R$ be as in the previous section. We define an order $\le$ on the monomials in $R$ as follows. First, we order the variables $x_{i,j}$ lexicographically by comparing the second index first: that is, $x_{i,j} < x_{k,\ell}$ if $j < \ell$ or $j=\ell$ and $i<k$. We then order monomials by lexicographically comparing their exponents. This is a well-ordering of the monomials and compatible with multiplication. We write $\ini(f)$ for the initial term of a non-zero element $f \in R$ and $\ini(I)$ for the initial ideal associated to an ideal $I \subset R$.

\begin{lemma}
We have $\ini(I) \cap R_n = \ini(I \cap R_n)$.
\end{lemma}

\begin{proof}
It is clear that $\ini(I \cap R_n) \subset \ini(I) \cap R_n$, so let us prove the reverse containment. The ideal $\ini(I) \cap R_n$ is monomial, so it suffices to show that if $f \in I$ and $\ini(f) \in R_n$ then $f \in R_n$. But this is clear from how we ordered the variables: indeed, if $\ini(f)=m \in R_n$ then no monomial appearing in $f$ can contain a variable of the form $x_{i,j}$ with $j>n$, for then that monomial would exceed $m$ in our ordering and contradict $m$ being the initial term, and so it follows that $f \in R_n$.
\end{proof}

\begin{lemma} \label{lem:hilbinit}
We have $H_I(s,t)=H_{\ini(I)}(s,t)$.
\end{lemma}

\begin{proof}
The coefficient of $s^n$ in $H_I(s,t)$ is equal to $H_{I \cap R_n}(t)$. It is a standard fact that passing to the initial ideal does not affect Hilbert series, and so this is equal to $H_{\ini(I \cap R_n)}(t)$. By the lemma, this is equal to $H_{\ini(I) \cap R_n}(t)$, which is the coefficient of $s^n$ in $H_{\ini(I)}(s,t)$.
\end{proof}

\begin{theorem}
Let $I$ be an $\Inc(\bN)$-stable ideal in $R$. Then $H_I(s,t)$ is a rational function.
\end{theorem}

\begin{proof}
This follows from the previous lemma and Theorem~\ref{thm:monomial}. (Note that our monomial ordering is compatible with the action of $\Inc(\bN)$, and so $\ini(I)$ is still $\Inc(\bN)$-stable.)
\end{proof}

\section{An algorithm for Hilbert series}\label{s:algorithm}

We now describe an algorithm for computing $H_I(s,t)$ for an $\Inc(\bN)$-stable ideal $I$ as above. We first recall some additional background material. Suppose that $\cL$ is a regular language. Then there is a finite-state automaton $\cA$ that accepts precisely the words in $\cL$, see \cite[Ch.~2]{hopcroftullman}. Fix such an $\cA$, and suppose that it has $N$ states. For $\ell \in \Sigma$ let $M_{\cA,\ell}$ be the associated transition matrix for $\cA$. This is the 0-1, left-stochastic $N \times N$ matrix with~1 in entry $(i,j)$ if there is edge labeled by $\ell$ from state $j$ to state $i$. Let $\BF e_1 \in K^n$ be the basis vector for the initial state, and let $\BF u = \sum_{i \in \mathcal{F}} \BF e_i \in K^n$ be the sum of the basis vectors corresponding to the accept states $\mathcal{F}$. Then for a word $w=w_1 \cdots w_n$, we have
\begin{displaymath}
{}^t\BF u M_{\cA,w_n} \cdots M_{\cA,w_1} \BF e_1 = \begin{cases}
1 & \text{if $\cA$ accepts $w$} \\
0 & \text{if $\cA$ rejects $w$.}
\end{cases}
\end{displaymath}
Let $\rho \colon \Sigma^{\star} \to \cM$ be a weight function, where $\cM$ is the set of monomials in $t_1, \ldots, t_k$. Summing the above expression over all words, we find
\begin{equation}
\label{eq:hs}
\begin{aligned}
H_{\cL,\rho}(t_1,\ldots,t_k)
&= \sum_{w \in \cL} \rho(w)
= \sum_{n\geq 0} {}^t\BF u\bigg(\sum_{\ell \in \Sigma} \rho(\ell)M_{\cA,\ell}\bigg)^n\BF e_1 \\
&= {}^t\BF u\bigg(\Id - \sum_{\ell \in \Sigma} \rho(\ell)M_{\cA,\ell}\bigg)^{-1}\BF e_1.
\end{aligned}
\end{equation}
Thus the generating function for $\cL$ can be computed directly from the automaton $\cA$.

The following is our algorithm for computing $H_I(s,t)$, given as input a set of elements $f_1, \ldots, f_r$ of $I$ whose $\Inc(\bN)$-orbits generate $I$:
\begin{enumerate}
\item First compute the initial ideal of $I$. This can be done using standard equivariant Gr\"obner basis techniques. We suppose that $m_1, \ldots, m_s$ are monomials whose $\Inc(\bN)$ orbits generate the initial ideal.
\item Next construct a regular expression for the language $\cL=\bm^{-1}(\langle m_1, \ldots, m_s \rangle)$. We note that \eqref{eq:regex} is essentially a regular expression for $\bm^{-1}(\langle m \rangle)$ (and is obviously constructed algorithmically from $m$), and a regular expression for $\cL$ can be obtained by ``or-ing'' the regular expressions for the various $\bm^{-1}(\langle m_i \rangle)$.
\item From the regular expression for $\cL$, construct an automaton $\cA$ that accepts $\cL$. It is well-understood how to algorithmically pass from a regular expression to an automaton, see \cite[Ch.~2]{hopcroftullman}.
\item Finally, compute the Hilbert series from the automaton via \eqref{eq:hs}, using the weight function from Proposition~\ref{prop:Igfunc}. This really computes the Hilbert series of the initial ideal, but this coincides with the Hilbert series of the original ideal $I$ by Lemma~\ref{lem:hilbinit}.
\end{enumerate}

\begin{example}
Let $r = 1$ and $I = \langle x_{1,0}^2 \rangle$. The language $\bm^{-1}(I)$ is detected by the regular expression
\begin{displaymath}
(\xi_1 \vert \tau)^{\star} \xi_1 \xi_1 (\xi_1 \vert \tau)^{\star}
\end{displaymath}
and by the automaton
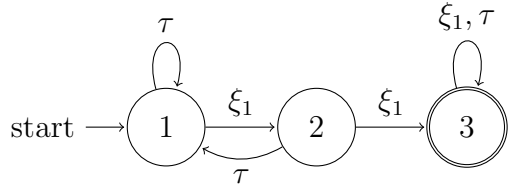
\begin{center}
\begin{tikzpicture}[shorten >=1pt,node distance=2cm,on grid,auto] 
   \node[state,initial] (q_0)   {$1$}; 
   \node[state] (q_1) [right=of q_0] {$2$}; 
   \node[state,accepting] (q_2) [right=of q_1] {$3$}; 
    \path[->] 
    (q_0) edge  node {$\xi_1$} (q_1)
          edge [loop above] node {$\tau$} ()
    (q_1) edge  node {$\xi_1$} (q_2)
          edge [bend left]  node {$\tau$} (q_0)
    (q_2) edge [loop above] node {$\xi_1,\tau$} ();
\end{tikzpicture}
\end{center}
where the first two states are rejecting and the last accepting. The automaton has transition matrices
\begin{displaymath}
M_{\cA,\tau} = \begin{bmatrix}1&1&0\\0&0&0\\0&0&1\end{bmatrix}, \qquad
M_{\cA,\xi_1}    = \begin{bmatrix}0&0&0\\1&0&0\\0&1&1\end{bmatrix}.
\end{displaymath}
We have $\BF e_1 = (1,0,0)$ and $\BF u = (0,0,1)$, and so
\begin{displaymath}
H_I(s,t) = {}^t\BF u(\Id - sM_{\cA,\tau} - tM_{\cA,\xi_1})^{-1}\BF e_1 = \frac{t^2}{(1-s-t)(1-s-st)}. \qedhere
\end{displaymath}
\end{example}

We implemented functions constructing  automata corresponding to monomial ideals in $R$ and computing their Hilbert series in Macaulay2~\cite{wwwM2}. These along with some examples are posted at \url{http://rckr.one/eHilbert.html}.

\section{Hilbert series of modules} \label{s:module}

Let $M$ be a graded $R$-module equipped with a compatible action of $S_{\infty}$ that is generated by the $S_{\infty}$ orbits of finitely many elements.\footnote{For technical reasons related to our uses of \cite{infrank} below, we assume that every element of $M$ is stabilized by a subgroup of $S_{\infty}$ of the form $\Aut(\{n,n+1,\ldots\})$. This is automatic if $M$ is an ideal in $R$.} A natural problem is to define a notion of Hilbert series for $M$ and extend Theorem~\ref{mainthm} to this setting.

One can generalize the definition of $H_I$ as follows. Let $G(n) \subset S_{\infty}$ be the subgroup consisting of permutations that fix each of the elements $0, \ldots, n$. Then $I \cap R_n$ is identified with the invariants $I^{G(n)}$. Thus in the definition of $H_I$ we can simply replace $I \cap R_n$ with $M^{G(n)}$ to obtain a definition for $H_M$, i.e.:
\begin{displaymath}
H_M(s,t) = \sum_{n \ge 0} H_{M^{G(n)}}(t) s^n.
\end{displaymath}
This is a perfectly well-defined series, and so one can certainly study it and investigate its rationality properties. However, as a definition of Hilbert series it is fatally flawed: formation of $G(n)$ invariants is not exact, and so the above quantity is not additive in short exact sequences of $R$-modules. (For example, if $I=\langle x_{1,1}-x_{1,0} \rangle \subset R$ then $R^{G(-1)}$ is the set of constants where $G(-1)=S_{\infty}$.  Meanwhile $(R/I)^{G(-1)} = R/I \cong K[y]$, and therefore $(R/I)^{G(-1)} \ne R^{G(-1)}/I^{G(-1)}$.)

There are various ways one could try to fix this problem: one could substitute invariants with derived invariants, which is known to be well-behaved by \cite[\S 6.4.4]{infrank}, or with coinvariants, which is known to be exact by \cite[\S 6.2.11]{infrank}. However, the best series to study is probably
\begin{displaymath}
H_M = \sum_{n \ge 0} [M_n] t^n,
\end{displaymath}
where $[M_n]$ is the class of the $S_{\infty}$-representation $M_n$ in the Grothendieck group of finitely generated algebraic representations (in the sense of \cite[\S 6]{infrank}). Any reasonable notion of Hilbert series for $M$ should factor through the above definition. We note that the Grothendieck group in question is identified with the ring of symmetric functions $\Lambda$, so that above series can be considered as a power series in $t$ with coefficients in $\Lambda$. We believe there should be some sort of rationality theorem for $H_M$, but leave this as an open problem.

\end{document}